\documentclass{tran-l}
\usepackage{graphics}
\usepackage{mathtools}
\usepackage{tipa}
\usepackage{amssymb}
\usepackage{amsmath}
\usepackage{color}

\usepackage{amsmath}

\makeatletter
\def\article@logo{}
\def\@setcopyright{}
\def\@adminfootnotes{%
  \let\@makefnmark\relax
  \let\@thefnmark\relax
  \ifx\@empty\thankses\else \@footnotetext{%
    \def\par{\let\par\@par}\@setthanks}%
  \fi
}
\makeatother

\newtheorem{thm}{Theorem}[section] 
\newtheorem{lem}[thm]{Lemma}     
\newtheorem{cor}[thm]{Corollary}
\newtheorem{prop}[thm]{Proposition}

\theoremstyle{definition}
\newtheorem{defn}[thm]{Definition}
\theoremstyle{remark}
\newtheorem{rem}[thm]{Remark}
\newtheorem{exa}[thm]{Example}
\numberwithin{equation}{subsection}

\newcommand{\DE}{\Delta}
\newcommand{\de}{\delta}

\newcommand{\te}{\vartheta}

\newcommand{\U}{\mathcal{U}}
\newcommand{\HH}{\mathcal{H}}

\newcommand{\R}{\mathbb{R}}

\newcommand{\N}{\mathbb{N}}

\newcommand{\FI}{\mathfrak{I}}

\newcommand{\FM}{\mathfrak{M}}

\newcommand{\FIG}{\mathfrak{IG}}

\newcommand{\hh}{{H}}

\newcommand{\CG}{\mathcal{G}}
\newcommand{\CGS}{\mathcal{GS}}

\begin{document}

\title[A finite Hausdorff dimension for graphs]
 {}

 \title[A finite Hausdorff dimension for graphs]
 {A finite Hausdorff dimension for graphs}

\author{ Juan M. Alonso}

\address{Dpto. de Matem\'atica, Universidad Nacional de San Luis, and Facultad de Ciencias Exactas y Naturales, UN de Cuyo, Mendoza, Argentina.}

\email{jmalonso@unsl.edu.ar}

\thanks{This work was completed with the support of an UNCuyo Research Scholarship (SeCTyP)}








\begin{abstract}The classical Hausdorff dimension of finite or countable metric spaces is zero. Recently, we defined a variant, called \emph{finite Hausdorff dimension}, which is not necessarily trivial on finite metric spaces. In this paper we apply this to connected simple graphs, a class that provides many interesting examples of finite metric spaces. There are two very different cases: one in which the distance is coarse (and one is doing Graph Theory), and another case in which the distance is much finer (and one is somewhere between graphs and finite metric spaces). 
\end{abstract}

\maketitle

\section{Introduction}

A \emph{finite Hausdorff dimension} for finite metric spaces, denoted $ \dim_{fH}$, was defined in \cite{alo1}. In contrast to the classical notion, finite Hausdorff dimension is not necessarily zero on finite sets. In this paper we apply this theory to (finite, connected, simple) graphs. There are two essentially different cases, depending on how the graph is metrised. In the simplest case, when the distance is given by hop count, we are doing Graph Theory. The general case is a medley between graphs and (finite) metric spaces. 

Both cases have applications. The general case is used to compute the finite dimension of plants and to show that, when considered as a function of time, it is intimately related to the ontogenesis of the plant (\cite{alo2}). The finite dimension of graphs metrised according to hop count can be used to study GlycomeDB (\cite{glyDB}), an open data base containing 40,000+ computed carbohydrate structures. We cluster the glycans according to finite dimension (there are only 110 different values for the dimension), and study the biological characteristics of the glycans in the clusters (work in progress).

In Section~\ref{s:review} we review the definition of the classical Hausdorff dimension and indicate the changes necessary to define finite Hausdorff dimension. In Section \ref{s:Mgraphs} we define metric structures on graphs, a theme to be continued in sections \ref{s:intrMET} and \ref{s:intrinsic}. Sections \ref{s:cat} to \ref{s:prod} deal with graphs endowed with the simplest possible metric (hop count), the one commonly used in Graph Theory (see, e.g. \cite{ham}). In Section \ref{s:cat} we define the finite Hausdorff dimension of graphs. In Section \ref{s:extremal} we calculate the extreme values taken by the finite dimension in the class of graphs with a fixed number of vertices. In Section \ref{s:prod} we deal with the finite dimension of products of graphs. 

We study general \emph{intrinsic metric graphs} in Section \ref{s:intrMET}, and show that every positive real number is the finite dimension of an intrinsic graph, while the values of the finite dimension of graphs with hop count is dense in the interval $[1,\infty)$. Finally, in Section \ref{s:intrinsic} we define \emph{intrinsic} (or \emph{length}) \emph{finite metric spaces}, and show that the finite metric spaces obtained from graphs are precisely the intrinsic finite metric spaces. 


\section{Review of finite dimension}\label{s:review}

 To make the paper more self-contained, we summarise here the definition and some properties of the finite Hausdorff, and \emph{finite Box-counting}, dimensions. For details we refer the reader to \cite{alo1}.  The dimensions are defined for arbitrary \emph{finite} metric spaces and, in contrast to their classical counterparts, are not trivial for these spaces. Note that "finite" in "finite dimension" refers to the fact that both dimensions are defined \emph{only} for finite spaces. The values taken by  the dimensions, on the other hand, can be any non-negative real number, or infinity.




Given a finite metric space $X$, henceforth simply referred to as a metric space, there are three parameters of interest to us: the smallest length among its points, denoted $\de$ or $\de(X)$, the diameter $\DE$ or $\DE(X)$ and the \emph{2-covering diameter} $\nabla$ or $\nabla(X)$. The first two are well-known; the third, to be defined presently, can be interpreted as the smallest real number $\nabla$ with the property that every point in the space has a (different) neighbour at distance at most $\nabla$. Notice that in general
\[
0<\de(X)\leq \nabla(X)\leq \DE(X).
\]

The definition of finite Hausdorff dimension follows closely the classical definition (see, e.g. \cite{fal}), except for a couple of crucial changes. Instead of just coverings we consider \emph{2-coverings}, defined as a family $\U=\{U_1,\dots,U_n\}$ of subsets that cover $X$ and have positive diameter. Having positive diameter is equivalent to the condition that each $U_i$ has at least two elements. The \emph{diameter} of $\U$ is the largest diameter of the $U_i$, and the \emph{2-covering diameter} $\nabla(X)$ is the smallest diameter of a 2-covering of $X$.

Given $\U$ and $s\geq 0$, let
\[
H^s_\U(X):=\sum_{i=1}^n \DE(U_i)^s.
\]
For $\eta\geq\nabla$,  set$H^s_\eta(X):=\min\{H^s_\U\,|\,  \DE(\U)\leq \eta\}$, and let $H^s(X):=H^s_\nabla(X)$. The function $H^s$ is an analog of Hausdorff's $s$-outer measure $\HH^s$. In the classical case, there is a value $s_0$ with the property that $\HH^s=0$ for $s>s_0$, and $\HH^s=\infty$ for $s<s_0$. The Hausdorff dimension is then defined to be $s_0 :=\dim_H $. There is no such value in the finite case, so we "manufacture" one by considering the equation
\begin{equation} \label{e:finHAU}
H^s(X)=\DE(X)^s,
\end{equation}
and solving for $s$. It turns out that this equation has a unique solution iff $\nabla(X)<\DE(X)$. The finite Hausdorff dimension of $X$, denoted $\dim_{fH}(X)$, is defined to be $0$ when $X$ is a point, $\infty$ when $\nabla(X)=\DE(X)$, and to be the unique solution of the above equation, a positive real number, when $\nabla(X)<\DE(X)$.

As in the classical case, $\dim_{fH}$ is not so easy to compute. We have also defined a finite Box-counting dimension, denoted $\dim_{fB}$. Just as in the classical case, $\dim_{fB}$ is easier to compute than $\dim_{fH}$ and, moreover, there is an explicit formula:
\begin{equation}\label{e:Box}
\dim_{fB}(X)=\frac{\ln N_{\nabla(X)}}{\ln\frac{\DE(X)}{\nabla(X)}}
\end{equation}
where $N_\nabla=N_{\nabla(X)}$ is the smallest number of elements of a 2-covering of $X$ of covering diameter $\nabla(X)$. A metric space $X$ is called \emph{locally uniform} (\cite{alo1}, Def. 5.4) when $\de(X)=\nabla(X)$. For these spaces $\dim_{fH}=\dim_{fB}$, and the dimension can be computed with the  explicit formula (\ref{e:Box}).

Both finite dimensions behave well with respect to \emph{H\"older equivalences}. Recall that a function $h:M\to M'$ is called $(r,\beta)$-H\"older equivalence if $d'(h(x),h(y))= r d(x,y)^\beta$, for all $x,y\in M$, and some fixed $r,\beta >0$. In the special case $\beta=1$, $h$ is called a \emph{similarity}, or $r$-similarity. When $\beta=1=r$, $h$ is an \emph{isometry}. For H\"older equivalences we have (\cite{alo1}, Thm 3.16):
\[
\beta \dim_{fH}(h(M)) =\dim_{fH}(M),
\]
and there is a similar formula for finite Box-dimension (\cite{alo1}, Thm. 4.12). Hence both finite dimensions are invariant under similarities and, in particular, under isometries.

\section{Intrinsic metrics in graphs} \label{s:Mgraphs}

To define the finite dimension of a graph we need to turn it into a metric space. In this section we examine a way to do this. 

Recall that graphs in this paper are \emph{finite}, \emph{simple} and \emph{connected}. A simple graph is the same as a $1$-dimensional simplicial complex. A graph $G=(V,E)$ consists of a set $V=V(G)$ of vertices and a set $E=E(G)$ of edges; every edge is a non-ordered pair of vertices. Vertices are \emph{adjacent} when they are endpoints of an edge. By a \emph{metric structure on a graph} $G$ we mean a metric structure on its set of vertices:

\begin{defn}
$G$ is a \emph{metric graph} if there is a distance $d$ defined on the set of vertices, i.e. if $(V(G),d)$ is a metric space. 
\end{defn}

Note that we exclude the edges from the metric structure, although it is possible, and not difficult, to metrise the whole graph.

\emph{Intrinsic metrics} are defined using the graph structure via paths and their lengths, as follows. A function $g:E(G)\to \R$ is called an \emph{edge-length} function if $g(e)>0$, for all $e\in E(G)$. A \emph{path} in $G$ joining $v_0$ to $v_k$ consists of a sequence of vertices $v_0,v_1,\dots,v_k$, such that $v_{i-1}$ and $v_i$ (for $i=1,\dots,k$) are adjacent joined, say, by edge $e_i$. We denote this path $p=e_1\dots e_n$. Given a pair $(G,g)$ we can associate a \emph{length} $\ell_g(p)$ to paths, and use $\ell_g$ to define $d_g$ on $V(G)$:

\[
\ell_g(p):=\sum_{i=1}^ng(e_i),\quad d_g(v_0,v_1):=\min\{\ell_g(p)\,|\, p\textrm{ is a path from }v_0\textrm{ to }v_1\}. 
\]

\begin{lem} \label{ }
Given $G=(V,E)$ and $g$ as above, $(V,d_g)$ is a metric space, and $d_g$ is called the \emph{intrinsic metric} defined by $g$.
\end{lem}

\begin{defn}
Suppose that $G$ is a metric graph with distance $d$, and $G\neq K_n$. We call $G$ an \emph{intrinsic metric graph} if $d=d_g$ for some edge-length function $g$. 
\end{defn}

We abuse language and simply say that $G$ is intrinsic, if it has an intrinsic metric (and $G\neq K_n$), i.e. if there is a metric $d_g$ on $V(G)$, for some $g$.

\begin{rem}
Recall that $K_n$, the \emph{complete graph} on $n$ vertices, is the graph in which every pair of vertices is adjacent. We exclude the case $G=K_n$ because \emph{any} finite metric space $(F,d)$ is isometric to $(V(K_n),d_g)$, i.e., to the "intrinsic" metric graph $K_n$ with metric $d_g$. Indeed, suppose that $F$ has $n$ elements, and define an edge-length function $g$ on $K_n$ by $g(\{v_0,v_1\}):=d(\{v_0,v_1\})$. Then $d_g=d$, and the spaces are isometric. \end{rem}

When $g$ is constant $d_g$ is rather coarse, as the following result shows.

\begin{lem}\label{l:isometry=iso}
Suppose given graphs $G, H$ with edge-length functions $g:E(G)\to \R$, $h:E(H)\to \R$. If $g,h$ are constant, say $g\equiv a$ and $h\equiv b$ for some $a,b>0$, then the following conditions are equivalent:
\begin{enumerate}
  \item There exists an $r$-similarity $\varphi:(V(G),d_g)\to (V(H),d_h)$, where $r=b/a$.
  \item There exists a graph isomorphism $\bar{\varphi}:G\to H$.
  \item There exists a simplicial isomorphism $\tilde{\varphi}:G\to H$.
\end{enumerate}
\end{lem}

\begin{proof}
(i) implies (ii). The hypothesis means that $\varphi$ maps the vertices of $G$ onto those of $H$, and $d_h(\varphi(v),\varphi(v'))=rd_g(v,v')$. It follows that $\varphi$ is bijective and vertices at distance $a$ are mapped to vertices at distance $b$. In other words, adjacent vertices are mapped to adjacent vertices, since adjacency in $G$ is equivalent to the vertices having distance $a$, and similarly for $H$. Extending $\varphi$ to $E(G)$ by $\bar{\varphi}(\{v,v'\}):=\{\varphi(v),\varphi(v')\}$ gives the desired graph isomorphism. We leave the proof of the rest of the assertions to the reader.
\end{proof}

\begin{cor}\label{c:invar}
In the special case when $a=b$, we have the following equivalent conditions:
\begin{enumerate}
  \item There exists an isometry $\varphi:(V(G),d_g)\to (V(H),d_h)$.
  \item There exists a graph isomorphism $\bar{\varphi}:G\to H$.
  \item There exists a simplicial isomorphism $\tilde{\varphi}:G\to H$.
\end{enumerate}
\end{cor}

\begin{rem}
The lemma and its corollary show that the metric structure defined by constant edge-length maps is so coarse that similarity, isometry, graph isomorphism and simplicial isomorphism coincide. 
\end{rem}

We follow the usual definitions of Graph Theory and metrise the vertices with $d_g$ for $g$ constant with value 1. In this case $\de(G)=1=\nabla(G)$, so that graphs with this metric are locally uniform spaces. For later reference we record this fact:

\begin{cor}\label{c:locunif}
When $g\equiv 1$, $(G,d_g)$ is a locally uniform metric space, i.e.
$\delta(G)=\nabla(G)=1$.
\end{cor}

\begin{rem} \label{r:LUnotCONST}
Note that intrinsic graphs can be locally uniform \emph{whithout} $g$ being constant and equal 1. For instance, consider $G=P_4$, the path with four vertices and three edges $e_1,e_2,e_3$, with $g(e_1)=g(e_3):=1$, and $g(e_2):=x\geq 1$. With $d_g$, we have $\de(G)=1=\nabla(G)$, so that $G$ is locally uniform.
\end{rem}

We leave the study of the much finer distances obtained when $g$ is not constant to the last two sections.

\section{Finite Hausdorff dimension in Graph Theory} \label{s:cat}

\begin{defn} \label{d:finDIM}
The \emph{finite Hausdorff dimension of a metric graph $G$} is defined by
\[
\dim_{fH}(G):= \dim_{fH}(V(G),d).
\]
\end{defn}

In this section and in sections \ref{s:extremal} and \ref{s:prod}, we consider \emph{exclusively} the "default", graph-theoretical case, so that $\dim_{fH}(G):= \dim_{fH}(V(G),d_{g\equiv 1})$.

By Corollary \ref{c:invar},
\begin{cor}
When $(G,g\equiv 1)$, finite Hausdorff dimension is a graph invariant.
\end{cor}

The case $g\equiv 1$ is very special from the point of view of finite Hausdorff dimension. By Corollary~\ref{c:locunif}, these graphs are locally uniform spaces and there is an explicit formula (\ref{e:Box}) for the dimension. In this situation, the finite Hausdorff dimension will be simply called \emph{finite dimension} and denoted $\dim_f(G)$.

\begin{cor}\label{c:formula}
When $(G,g\equiv 1)$, the finite dimension of $G$ is given by the formula:
\begin{equation}\label{e:formula}
\dim_f(G)=\frac{\ln N(G)}{\ln\DE(G)}
\end{equation}
where $N(G)$ is the minimal number of elements of 2-coverings of diameter $\nabla=1$ of the graph.
\end{cor}

\noindent It turns out that $N(G)$ coincides with the well-known graph invariant \emph{(vertex) clique covering number} $\vartheta(G)$, as the next result shows. By Karp's classical work on complexity \cite{karp}, clique cover is NP-complete. Hence, computing $\dim_f$ is also NP-complete.


\begin{lem}
For any graph $G$, 
$
N(G)= \vartheta(G), 
$
where $\vartheta(G)$ denotes the \emph{(vertex) clique covering number} of $G$.
\end{lem}

\begin{proof}
Recall an $m$-clique is a set $W\subseteq V(G)$ of $m$ vertices that generates a complete graph $K_m$. The general definition of finite Hausdorff dimension requires to cover the vertices by 2-coverings. Translated to our situation, an element of a 2-covering is precisely a set of diameter 1. Observe that $K_m$ has diameter 1 iff $m\geq 2$. Thus, to compute $\vartheta(G)$ one covers with $m$-cliques for $m\geq 1$, whereas to compute $N(G)$, we use only cliques with $m\geq 2$. It is clear by definition that $\te(G)\leq N(G)$. On the other hand, given an arbitrary clique covering with $\te(G)$ elements, we can add a new vertex to each clique consisting of a single vertex, and convert the covering into a 2-covering of diameter one. Thus $N(G)\leq \te(G)$. The proof is complete.
\end{proof}

We can reformulate Corollary~\ref{c:formula} to express the result in terms of purely graph theoretic concepts:

\begin{cor}\label{c:formula2}
The finite dimension of a graph $G$ is given by the formula:
\begin{equation}
\dim_{f}(G)=\frac{\ln \vartheta(G)}{\ln \DE(G)}
\end{equation}
where $\te(G)$ denotes the graph's clique covering number, and $\DE(G)$ the diameter of the graph.
\end{cor}

\section{Extremal graphs} \label{s:extremal}

In this section we consider the extremal  values that $\dim_f$ can take in the set of graphs with $n$ vertices. 

Note first that $V$ has a \emph{focal point} (see \cite{alo1}) iff $\DE(V)=1$, that is, iff $V$ is a complete graph $K_n$. Consequently, $\dim_f(G)=\infty$ iff it is a complete graph (with at least one edge). The graph with only one vertex and no edges is the only graph with $\dim_f(G)=0$. For all other graphs, the dimension is a positive real number (see \cite{alo1} for details).

The question is trivial for $n<4$: for $n=1$ the only graph with one vertex consists of a point and has dimension 0. For $n=2$ there is only $K_2$, and its dimension is infinity. The graphs with $n=3$ are $K_3$ and $P_3$, with dimensions infinity and 1, respectively.

Let $p,q\geq 2$ be integers and $K_p,K_q$ the corresponding complete graphs on $p$ [resp. $q$] vertices. Let $A\subseteq V(K_p)$ denote a non-empty, \emph{proper} subset of the vertices of $K_p$, and similarly for $B\subseteq V(K_q)$. A graph in $L_{p,q}$ consists of the disjoint union of $K_p$ and $K_q$, together with edges $e_1,\dots,e_r$, $r\geq 1$, such that one end-point of each $e_i$ belongs to $A$, and the other to $B$. Given a graph $G$ in $L_{p,q}$, we usually abuse language and write directly $L_{p,q}$, instead of $G$. Clearly, any graph in $L_{p,q}$ has $p+q$ vertices. The only  graph in $L_{2,2}$ is the path $P_4$. However, for $p+q\geq 5$, the class $L_{p,q}$ has more than one graph. The next result characterises the set of graphs $L_{p,q}$ as those having finite dimension $\ln 2/\ln 3$. 

\begin{prop}\label{p:L-pq}
Let $G$ be a graph in $L_{p,q}$, for some $p,q\geq 2$. Then
\begin{enumerate}
  \item $G$ is connected.
  \item $\DE(G)=3$.
  \item $\dim_{f}(G)=\ln 2/\ln 3$.
  \item A graph $H$ with $p+q$ vertices is in $L_{p,q}$ iff $N(H)=2$ and $\DE(H)=3$.
\end{enumerate}
\end{prop}

\begin{proof}
Both $K_p$ and $K_q$ have diameter 1 because $p,q\geq 2$. Hence any vertex in $K_p$ is adjacent to all vertices of $A$, and at least one of these is adjacent to a vertex of $B$ which is, in its turn, adjacent to all other vertices of $K_q$. It follows that $G$ is connected and, moreover, has diameter $\leq 3$. 

We now claim that if $v\in V(K_p)\setminus A$ and $w\in V(K_q)\setminus B$, then $d(v,w)=3$. This will prove (ii). Any path from $v$ to $w$ must contain one of the $e_i$. Since the end-points of $e_i$ lie one in $A$ and the other in $B$, we need at least two more edges, as desired. 

To see (iii) we need only show that $N(G)=2$, since we already know (ii). Clearly, $\{V(K_p),V(K_q)\}$ is a minimal 2-cover for $G$.

To prove (iv) observe that the condition is necessary, by (ii) and (iii). To see sufficiency, suppose that $H$ is a graph with $N(H)=2$ and $\DE(G)=3$. Let $\{U_1,U_2\}$ be a 2-cover of $H$ by cliques. Thus, $|U_1|,|U_2|\geq 2$, and $\DE(U_1)=1=\DE(U_2)$. If $U_1\cap U_2\neq \emptyset$, then $\DE(H)\leq 2$. Hence $U_1\cap U_2=\emptyset$. Note that $U_j$ ($j=1,2$) generates $K_{|U_j|}$ in $H$, i.e. $H$ contains disjoint copies of the complete graphs $K_{|U_j|}$. Since $H$ is connected, there is at least one edge joining a vertex of $U_1$ to a vertex of $U_2$. Hence, if $A\subset U_1$ is the set of vertices adjacent to a vertex of $U_2$ and, similarly, $B\subset U_2$ consists of the vertices adjacent to a vertex of $U_1$, it follows that neither $A$ nor $B$ is empty. Moreover, both must be proper subsets since otherwise $\DE(H)<3$. It follows that $H\in L_{p,q}$, for $p=|U_1|,\,q=|U_2|$, as desired. This completes the proof.
\end{proof}

\begin{thm}\label{t:extGRAPHS}
For any graph $G$ with $n\geq 4$ vertices, 
\begin{equation} \label{e:withK}
\frac{\ln 2}{\ln 3}\leq \dim_f(G)\leq \infty .
\end{equation}
If, moreover, $G\neq K_n$, then
\begin{equation}\label{e:finiteONLY}
\frac{\ln 2}{\ln 3}\leq \dim_f(G)\leq \frac{\ln (n-1)}{\ln 2} \cdot
\end{equation}
We have: $\dim_f (K_n)=\infty$, $\dim_f(L_{p,q})=\ln 2/\ln 3$, for any $p+q=n$, and $\dim_\hh(ST_n)=\ln (n-1)/\ln 2$.
\end{thm}

\begin{proof}
Note first that $N=N(G)\leq n-1$ (see \cite{alo1} Remark 2.10, or give a direct proof). If, moreover, $G\neq K_n$, then $N\geq 2$. By Corollary (\ref{c:formula}):
\begin{equation}
\label{e:bound1}
\frac{\ln 2}{\ln \DE}\leq \dim_f(G)\leq \frac{\ln (n-1)}{\ln \DE}\cdot
\end{equation}
Since $G\neq K_n$, we have $\DE\geq 2$, proving the upper bound of (\ref{e:finiteONLY}). To prove the lower bound, note that for a given $N\geq 2$, we have $\DE\leq 2N-1 $, 
so that 
\[
\min\{\ln 2/\ln \DE\,|\, 2\leq \DE\leq 2N-1 \}=\ln2/\ln(2N-1)
\]
and, hence, we need to compute $
\min \{{\ln 2}/{\ln (2N-1)}\, | \, N=2,3,\dots,n-1\},
$
which equals $\ln 2/\ln 3$. Finally, it is obvious that the extremal values correspond to the dimensions of $K_n$, $L_{p,q}$ and $ST_n$, respectively. This completes the proof.
\end{proof}

\begin{lem}\label{p:paths}
Let $T$ be a tree with $n\geq 4$ vertices. Then the following conditions are equivalent:
\begin{enumerate}
  \item $T=P_n$, the path on $n$ vertices.
  \item $\DE(T)=n-1$.
  \item $\dim_{f}(T)=\frac{\ln \lceil n/2\rceil}{\ln (n-1)}$.
\end{enumerate}
\end{lem}

\begin{proof}
We prove that (iii) implies (ii) and leave the other implications to the reader. Assume for contradiction that $\DE\leq n-2$. Let $\U=\{U_i\}_{i=1}^N$ denote any 2-covering of minimal cardinality $N=N(T)$. Now, $T$ has no triangles, so from $V(T)\subseteq \cup U_i$, we obtain:
\[
n=|V(T)|\leq \sum_1^N |U_i|=2N,
\]
i.e. $n\leq 2N$, or $N \geq \lceil n/2\rceil$. Now, when $n=2k$, the inequality says $k\leq N$, and when $n=2k+1$, we have $k\leq N-1$. From
\[
\frac{\ln \lceil n/2\rceil}{\ln (n-1)} = \frac{\ln N }{\ln \DE},\quad \text{ we get }\quad  N<\lceil n/2 \rceil.
\]
For $n=2k$, we have $k\leq N < k=\lceil n/2 \rceil$, and when $n=2k+1$, $k\leq N-1 < k$. In both cases we obtain a contradiction, as desired.
\end{proof}

\begin{thm}\label{t:extTREES}
Let $T$ be a tree with $n\geq 4$ vertices. Then 
\[
\frac{\ln \lceil n/2\rceil}{\ln (n-1)}\leq \dim_f(T)\leq \frac{\ln (n-1)}{\ln 2} .
\]
Moreover, the lower bound is the finite dimension of the path $P_n$, and the upper bound the dimension of $ST_n$. 
\end{thm}

\begin{proof}
By Theorem~\ref{t:extGRAPHS}, only the lower bound needs to be proved. Note first that $\dim_f(P_n)=\ln\lceil n/2\rceil /\ln(n-1)$. Let $T$ be a tree with smallest dimension; then $\dim_f(T)\leq \dim_f(P_n)$. If $T\neq P_n$, by Lemma~\ref{p:paths}, we may assume, for contradiction, that $\DE(T)\leq n-2$.  Since $N\geq \lceil n/2 \rceil$, we have:
\[
\dim_f(P_n)\geq\dim_f(T)= \frac{\ln N}{\ln \DE(T)} \geq \frac{\ln \lceil n/2\rceil}{\ln (n-2)}> \frac{\ln \lceil n/2\rceil}{\ln (n-1)},
\]
a contradiction. 
\end{proof}

\begin{rem}
In contrast to the case of general graphs, for trees there is only one graph achieving the minimun dimension, namely $P_n$. Note also that for $n\geq 5$, $\dim_f(L_{p,q}) < \dim_f(P_n)$.
\end{rem}

\section{Products of graphs}\label{s:prod}
There are several different definitions of products of graphs (see \cite{ham}), such as the \emph{strong product} $G\boxtimes F$, and the \emph{Cartesian product} $G\square F$. In metric terms, $G\boxtimes F$ is characterised by the fact that the vertex set $V(G\boxtimes F)=V(G)\times V(F)=V(G\square H)$, is given the metric $d_\infty$, whereas in $G\square F$ one gives the vertices the metric $d_1$. Recall that, given spaces $(X,d)$ and $(X',d')$, we have metric spaces $(X\times X',d_j)$, for $j=1,\infty$, where:

\[ 
d_j((x_1,x'_1),(x_2,x'_2)) :=
  \begin{cases}
      d(x_1,x_2)+d'(x'_1,x'_2)  & \quad \text{if } j=1\\
        \max\{d(x_1,x_2),d'(x'_1,x'_2)\}       & \quad \text{if } j =\infty\\
  \end{cases}
\]
for all $x_1,x_2\in X$, $x'_1,x'_2\in X'$. 

\begin{lem}\label{l:nuETC}
Given finite metric spaces $(X,d)$, $(X',d')$, consider $(X\times X',d_j)$, for $j=1,\infty$. The first three identities hold for any of the three product metrics:
\begin{enumerate}
\item $\delta(X\times X')=\min\{\delta(X),\delta(X')\}$,
\item $\nabla(X\times X')=\min\{\nabla(X),\nabla(X')\}$,
\item $\DE(X\times X',d_\infty)=\max\{\DE(X),\DE(X')\}$,
\item $\DE(X\times X',d_1)=\DE(X)+\DE(X')$,
\end{enumerate}
\end{lem}
\noindent We leave the proof to the reader.

\begin{thm}\label{t:STRONGprod}
Let $G,F$ denote finite simple graphs. Then
\[
\dim_{f}(G\boxtimes F)\leq \dim_{f}(G)+\dim_{f}(F).
\]
\end{thm}

\begin{proof}
By Lemma~\ref{l:nuETC}, $G\boxtimes F$ is locally uniform and $\DE(G\boxtimes F)=\max\{\DE(G),\DE(F)\}$. Let $\U=\{U_i|i=1,\dots,n\}$ denote a 2-covering of $G$ with $|\U|=N(G)$, and similarly $\U'=\{U'_j|j=1,\dots,m\}$ denote a 2-covering of $F$ with $|\U'|=N(F)$. Then $\U\times \U'$ is a 2-covering of $G\boxtimes F$, with $\DE(\U\times \U')=1$, and $|\U\times \U'|=N(G)N(F)$. Hence, $N(G\boxtimes F)\leq N(G)N(F)$. Then,
\[
\dim_{f}(G\boxtimes F)=\frac{\ln N(G\boxtimes F)}{\ln \DE(G\boxtimes F)} \leq \frac{\ln (N(G)N(F))}{\ln \max\{\DE(G),\DE(F)\}}\leq \frac{\ln N(G)}{\ln \DE(G)}+\frac{\ln N(F)}{\ln \DE(F)} 
\]
as desired.
\end{proof}

\begin{exa} \label{e:Prod=UNT<}
We show by examples that both equality and inequality can hold in Theorem~\ref{t:STRONGprod}. For $G=P_3$ and $F=P_3$, we have $\DE(P_3\boxtimes P_3)=2$ and $N(P_3\boxtimes P_3)=4$. Hence, $\dim_{f}(P_3\boxtimes P_3)=2=\dim_f(P_3)+\dim_f(P_3)$. When $G=P_3$ and $F=P_4$, we have $\DE(P_3\boxtimes P_4)=3$, $N(P_3\boxtimes P_4)=4$, so that $\dim_f(P_3\boxtimes P_4)=\ln 4/\ln 3\simeq 1.26$. On the other hand, $\dim_f(P_3)=1$, $\dim_f(P_4)=\ln 2/\ln 3\simeq 0.63$.
\end{exa} 

Theorem~\ref{t:STRONGprod} is false for the product $G\square F$, as shown by the following examples.

\begin{exa} \label{e:sqProd>}
In this example $\dim_{f}(G\square H)<\dim_f(G)+\dim_f(H)$. Let $G=P_4$ and $H=C_4$. Then $\dim_{f}(G\square H)=\ln 8/\ln 5\simeq 1.29$, $\dim_{f}(G)=\ln 2/\ln 3\simeq 0.63$, and $\dim_f (H)=1$.
\end{exa} 

\begin{exa} \label{e:sqProd+>}
In this example $\dim_{f}(G\square H)>\dim_f(G)+\dim_f(H)$. Let $G=H=L_{3,3}$. Then $\dim_{f}(G\square H)=\ln 12/\ln 6\simeq 1.38$, $\dim_{f}(G)=\ln 2/\ln 3\simeq 0.63$.
\end{exa}

\section{Intrinsic metrics} \label{s:intrMET}

In this section we consider intrinsic metrics derived from non-constant edge-length maps, and contrast them to the default case $g\equiv 1$. When $g$ is not constant there is no explicit formula like (\ref{e:formula}) to compute the finite Hausdorff dimension, and we are forced to use the general definition  (\ref{e:finHAU}).

Lemma \ref{l:isometry=iso} shows that when $g$ is constant, $d_g$ depends only on the graph. On the contrary, when $g$ is not constant, $d_g$ need not reflect much of the graph structure. We illustrate this point in different ways.

\begin{exa}
Consider the graphs $K_3$ (complete graph on 3 vertices) and $P_3$ (path on 3 vertices, of length 2) with edge-length maps $h,g$, respectively. Let $e_1,e_2,e_3$ denote the three edges of $K_3$, and suppose that $h(e_1)=h(e_2)=1$, $h(e_3)=5$, and that $g\equiv 1$. Then $(V(K_3),d_h)$ and $(V(P_3),d_g)$ are isometric but the graphs are not isomorphic, in contrast to Lemma~\ref{l:isometry=iso}. 
\end{exa}

Let $\FIG$ denote the set of all \emph{intrinsic metric graphs}, i.e. spaces $(V,d)$ where $V$ is the set of vertices of a graph $G\neq K_n$, and $d=d_g$, for some edge-lenght map $g:E(G)\to \R$. Let $\FIG_1\subset \FIG$ denote the subset of intrinsic metric graphs with constant $g\equiv 1$. Consider finite Hausdorff dimension as a function on these spaces:

\[
\dim_{fH}: \FIG\to (0,\infty) 
\]

\begin{thm}
The image $\dim_{fH}(\FIG_1)$ is countable and dense in $[1,\infty)$. By contrast, $\dim_{fH}(\FIG)=(0,\infty)$.
\end{thm}

\begin{proof}
Countability of the image of $\FIG_1$ is clear since there are only countably many isomorphism classes of finite graphs. To prove density we first show that the set $D:=\{\ln(k)/\ln(m)\,|\, k,m\in\N,\,\,  k, m\geq 2\}$ is dense in $(0,\infty)$; restricting the integers so that $k\geq m \geq 2$, we obtain density in $[1,\infty)$.

We show that every open interval contains an element of $D$. We seek integers $k,m$ such that $0\leq p/q \leq  \ln(k)/\ln(m)  \leq  r/q$, where we have taken the endpoints to be rational numbers, and $0\leq p<r$. Taking $m:=b^q$, for some integer $b\geq 2$, the above inequalities become $b^p \leq k \leq b^r$, hence it suffices to show that $b^r-b^p \geq 1$. But $b^r-b^p=b^p(b^{r-p}-1)\geq b^{r-p}-1 \geq b-1\geq 1$, as desired. Suppose, moreover, that $1\leq p/q$. Then, $k \geq b^p\geq b^q=m$, as desired. 

To complete the proof, it suffices to construct, for each $k\geq m\geq 2$, a finite graph $G$, with $T=T(G)=k$, and $\DE=\DE(G)=m$. We start with a path $P:=P_{m+1}$ of diameter $m$; $P$ has a 2-covering with $\lceil (m+1)/2\rceil$ elements. Set $c:= k-\lceil (m+1)/2\rceil\geq 1$. Suppose $P=e_1\dots e_m$, with successive vertices $v_0,v_1,\dots,v_m$, and attach a star $St_c$ to $P$, by identifying the center of the star with $v_1$. The resulting graph has finite dimension $\ln(k)/\ln(m)$, as desired.

We now prove that $\dim_\hh(\FIG)=[0,\infty)$.  Suppose that $t\in (0,\infty)$, and consider graphs $G=G(n,m,x)\in \FIG$, where $n,m\in\N$, where $n,m\geq 1$, and $x > 0$ is a real number. By definition, $G$ is a tree with $n+m+1$ edges, where $n$ edges are joined to a single vertex, say $v_0$, and $m$ edges are joined to another vertex $v_1$. The other edge, $e$, joins $v_0,v_1$. Define $g$ to be 1 on every edge except $e$, and $g(e):=x$. With $d_g$, $G$ is locally uniform, has diameter $2+x$, and $N(G)=n+m$. Solving $\dim_{fH}(G)=\ln(n+m)/\ln(2+x)=t$, gives $x=(n+m)^{1/t}-2$. Hence $x> 0$, provided we choose $n+m> 2^t$. With this choice, $\dim_{fH}(G)=t$, as desired.
\end{proof}

\vspace{3mm}


\section{Finite Metric Spaces}\label{s:intrinsic}

In this section we consider the set $\FM$ of \emph{all} finite metric spaces
 and define $\FI$, the subset of \emph{ intrinsic (or length) metric spaces}. We show that $\FI=\FIG$.

As in the continuous case (see, e.g. \cite{bbi}), we define \emph{arcs} (to distinguish them from continuous paths), their lengths and from this an intrinsic (or length) metric. Here are the details. Let $(F,d)\in \FM$ be arbitrary. An \emph{arc} $a$ in $F$ joining $x\neq y$ is a sequence of points $x=x_0,x_1,\dots,x_k=y$ of $F$. Two consecutive points of an arc are called \emph{segment}. The segments of $a$ are $s_i:=\{x_{i-1},x_i\}$; we can also write $a=s_1\dots s_k$. The \emph{count}, $c(a):=k\geq 1$, is the number of segments of $a$, and the \emph{length} $\ell(a)$ is given by:
\begin{equation} \label{}
\ell(a):=\sum_{i=1}^{k-1} d(x_i,x_{i+1}).
\end{equation}
An arc $a$ joining $x,y$ is a \emph{geodesic} if $\ell(a)=d(x,y)$.

An \emph{intrinsic metric space} is, roughly speaking, a space with the property that pairs of different points can be joined by geodesics, i.e. these are spaces for which the distance between points is the length of an arc joining them. If the space has $n$ elements
we can choose one geodesic por each pair, for a total of $n(n-1)/2$ geodesics. Moreover, we can choose these geodesics to be \emph{maximal}, in the sense that their count is maximal among all geodesics joining the points.

\begin{defn} \label{d:intrins}
A space $(F,d)\in \FM$  is an \emph{intrinsic} (or \emph{length}) space if there is a family $\CG$ of maximal geodesics, one for each (non-ordered) pair of different points of $F$, such that $\max\{ c(a) | a\in \CG   \}\geq 2$.
\end{defn}

\begin{rem} 
(a) Note that if $s_i:=\{x_{i-1},x_i\}$ is a segment of some $a\in\CG$, then the arc $\hat{a}:=s_i$ is the only maximal geodesic joining $x_{i-1},x_i$ (hence, $s_i\in \CG$). Indeed, if there is a geodesic $b$ joining $x_{i-1},x_i$, with $c(b)\geq 2$, then $a$ is not maximal, since replacing $s_i$ by $b$ in $a$,  would give a new geodesic $a'$ with $c(a')>c(a)$. 

(b) Let $\CGS:=\{a\in \CG|c(a)=1\}$ denote the set of \emph{geodesic segments}. Def.~\ref{d:intrins} requires that $\CG\setminus \CGS$, the set of "honest" arcs (i.e. arcs with count $>1$), be non-empty.
\end{rem}

\begin{exa}   
Consider a space $F$ with $n$ elements in which the distance between different points is $=1$. This space is not intrinsic since $\CG=\CGS$.
\end{exa}

\begin{prop} \label{p:invariance}
Suppose that $(F,d)$ and $(F',d')$ are isometric spaces. If $F$ is intrinsic then so is $F'$.
\end{prop}

\begin{proof} Suppose that $F$ is intrinsic with maximal geodesics $\CG$, and let $f:F\to F'$ be an isometry. For an arc $a=s_1\dots s_k \in \CG$, define $f(a):=f(s_1)\dots f(s_k)$, where $f(s_i):=\{f(x_{i-1}),f(x_i)\}$, if $s_i=\{x_{i-1},x_i\}$. Since $f$ is an isometry, $\ell(s)=\ell(f(s))$ and hence $\ell(a)=\ell(f(a))$, for all arcs. Hence the set $\CG':=\{\{f(a)\}|a \in \CG\}$ is a set of geodesics of $F'$, one for every pair of different points. Clearly, $f(a)$ is maximal if $a$ is.  Moreover, since $f$ preserves count, $\CG'\neq \CGS'$, as desired.
\end{proof}

We now show that intrinsic metric spaces can be represented by intrinsic metric graphs. Suppose $(F,d)$ is intrinsic with geodesics $\CG$, and geodesic segments $\CGS$. We define a graph $G=G(F,\CG)$, with edge-length map $g$, as follows. Set $V(G):=F$, $E(G):=\{\{x,y\}|\{x,y\} \in \CGS\}$, and $g(\{x,y\}):=d(x,y)$. Under the identity map, points correspond to vertices, geodesic segments to edges, and geodesic arcs in $F$ to geodesic paths of $G$. Observe that $G\neq K_n$.

\begin{prop}\label{p:MtoG}
$G=(F,\CG)$ is a connected graph, and $g$ is an edge-length function. Moreover, $(V,d_g)$ is an intrinsic metric graph, and $(F,d)$ is isometric to $(G,d_g)$.
\end{prop}

\begin{proof}
Clearly, $G$ is a simple graph. For any pair of different points $x,y\in F$, there is an arc $a$ in $\CG$ joining them. In $G$, this translates to a (geodesic) path  $p(a)$ joining $x,y\in G$, so $G$ is connected. Clearly, $g$ is an edge-length function, since $g(\{x,y\})>0$, and $G(F)$ is an intrinsic metric graph. Finally, $d(x,y)=\ell(a)=\ell(p(a))=d_g(x,y)$ shows that $F$ and $G(F)$ are isometric. The proof is complete.
\end{proof}

\begin{thm}
A metric space $(F,d)$ is intrinsic iff there is an intrinsic graph $G$ isometric to $F$. In other words, $\FIG=\FI$.
\end{thm}

\begin{proof} Suppose $(F,d)\in \FI$ is intrinsic, with geodesics $\CG$. By Prop.~\ref{p:MtoG}, $G=G(F,\CG)$ is the desired intrinsic graph.

Conversely, given $(G,g)\in \FIG$ we show that $(V,d_g)\in \FI$. Given $v_0\neq v_1\in V$, choose a geodesic path $p(v_0,v_1)$ in $G$, joining $v_0, v_1$.  Such path exists because the distance in $G$ is $d_g$. The \emph{count} of a path in $G$ is the number of its edges. We can then choose a geodesic path $p(v_0,v_1)$ which is maximal with respect to count. Call $\CG$ the set of such geodesics. We can obviously consider (maximal) geodesics in $G$ as (maximal) arcs in $V$. Since count of paths coincides with count of arcs, it follows that $\CGS=E(G)\neq \CG$, because $G\neq K_n$. This concludes the proof.
\end{proof}




\end{document}